%% file: Prescribed_negative_curvature2.tex
\documentclass[11pt]{amsart}
\usepackage[lmargin=1.2in,rmargin=1.2in]{geometry}
\usepackage{amsthm,amsfonts,amssymb,amsmath,comment,slashed,mathtools,bbold}
\usepackage[T1]{fontenc} 
\usepackage[utf8]{inputenc}
\usepackage[hidelinks,pdfpagelabels]{hyperref}   
\usepackage[capitalize]{cleveref} 
\usepackage{xcolor}
\usepackage[backend=bibtex,bibencoding=ascii,giveninits=true,doi=false,isbn=false,url=false,sorting=nyt]{biblatex} 
\renewbibmacro{in:}{} 

\bibliography{references.bib}

\theoremstyle{plain} 
 
\newtheorem{proposition}{Proposition}

\newtheorem{theorem}{Theorem}
\theoremstyle{remark}
\newtheorem{remark}{Remark}
\renewcommand{\d}{\mathrm{d}}

\numberwithin{equation}{section}

\author{Alessio Figalli}
\address{ETH Zürich, Department of Mathematics, Rämistrasse 101, 8092 Zürich, Switzerland}
\email{alessio.figalli@math.ethz.ch}

\author{Christoph Kehle}
\address{Institute for Advanced Study, School of Mathematics, 1 Einstein Drive, Princeton NJ, USA and ETH Zürich, Institute for Theoretical Studies, Clausiusstrasse 47, 8092 Zürich, Switzerland}

\email{c.kehle@ias.edu, christoph.kehle@eth-its.ethz.ch}

\title[The prescribed negative Gauss curvature problem
for graphs]{On the prescribed \\
negative Gauss curvature problem \\
for graphs}
\begin{document}
	
\maketitle
\thispagestyle{empty}

\begin{center}
{\it To Juan Luis Vázquez for his 75th birthday, with friendship and admiration.}
\end{center}

\begin{abstract}
We revisit the problem of prescribing negative Gauss curvature for graphs embedded in $\mathbb R^{n+1}$ when $n\geq 2$. The problem reduces to solving a fully nonlinear Monge--Amp\`ere equation that becomes hyperbolic in the case of negative  curvature. We show that the linearization around a graph with Lorentzian Hessian can be written as a geometric wave equation  for a suitable Lorentzian metric in dimensions $n\geq 3$. Using energy estimates for the linearized equation and a version of the Nash--Moser iteration, we show the local solvability for the fully nonlinear equation. Finally, we discuss some obstructions and perspectives on the global problem. 
\end{abstract}
\tableofcontents
\thispagestyle{empty}
\newpage
\section{Introduction}
The problem of showing the existence or non-existence of complete $n$-dimensional hypersurfaces $\mathcal{S}$ immersed in $\mathbb R^{n+1}$ with prescribed negative Gaussian curvature $K$ has a long and rich history. In the case of surfaces in $\mathbb R^3$, i.e.,  $n=2$, a classical result \cite{MR1500557} proved by Hilbert in 1901 states that there does not exist a complete smooth surface $\mathcal S$ of constant negative curvature $K$ immersed in $\mathbb R^3$. Hilbert's theorem was significantly strengthened by the later work of Efimov \cite{MR0167938}, who showed that there exists no complete $C^2$ surface $\mathcal S$ with negative curvature $K$  bounded away from zero, which is immersed in $\mathbb R^3$. In the positive direction, again for $n=2$, Hong (see \cite{MR1266477} and also \cite{MR2261749}) obtained a sufficient criterion in terms of the decay of the Gauss curvature measured by the geodesic distance to  admit an isometric embedding in $\mathbb R^3$. 
For higher dimensions $n\geq 3$, the question of the  existence of global hypersurfaces in $\mathbb R^{n+1}$ with prescribed negative Gauss curvature has remained widely open. 

In the paper at hand, we consider the problem of prescribed Gauss curvature for graphs in all dimensions $n\geq 2$.  In this setting the graph has to satisfy the prescribed Gauss curvature equation, see \eqref{eq:Gauss-curvature-equation} below, which is a fully nonlinear Monge--Amp\`ere equation. 
This setting has turned out to be particularly fruitful in the case of positive prescribed Gauss curvature \cite{MR3617963}  but also in the case of certain local questions where the Gauss curvature  changes sign, see e.g.\ \cite{MR2309171,MR3070566,MR3919494} and references therein. 

Let us assume that we have a smooth hypersurface $\mathcal S\subset \mathbb R^{n+1}$ given  as a graph, that is 
\begin{align*}
\mathcal S = \{ (x_1,\dots, x_n, u_{\mathcal S} (x_1, \dots, x_n) ) \colon x=(x_1,\dots, x_n) \in \mathbb R^n \} \subset \mathbb R^{n+1}.
\end{align*}
The normal to the hypersurface $\mathcal S $ is given by
\begin{align*}N = \frac{1}{\sqrt{1+|Du_{\mathcal S}|^2}} (-D u_{\mathcal S}, 1),
\end{align*} where $D$ denotes the Euclidean gradient in $\mathbb R^n$. For the first and second fundamental form we have the expressions
\begin{align*}
& I = \d x_1^2 + \dots \d x_n^2  + \bigg(\sum_{i=1}^n \partial_i u_{\mathcal S} \,\d x_i\bigg)^2, \;\qquad  II = \big(1+ |Du_{\mathcal S}|^2\big)^{-\frac 12}\sum_{i,j = 1, \dots n} D^2_{ij} u_{\mathcal S} \,\d x_i \d x_j.
\end{align*}
Noting that the Gauss curvature $K_{\mathcal S}$ is given by $K_{\mathcal S} = \frac{\det(II)}{\det(I)}$, the prescribed Gauss curvature equation for a graph $u_{\mathcal S}$ reads
\begin{align}\label{eq:Gauss-curvature-equation}
 \frac{\det D^2 u_{\mathcal S} }{\left(1+|Du_{\mathcal S}|^2\right)^{\frac{n+2}{2}}} = K_{\mathcal S}.
\end{align}
We see that the signature of the Hessian  $D^2u_{\mathcal S}$ is related to the sign of the  Gaussian curvature $K_{\mathcal S}$. More precisely, given positive $K_{\mathcal S}$, it is   natural to look for solution which are convex, i.e.,  with Hessian $D^2u_{\mathcal S}$ of signature $(+, \dots, +)$. In this case, \eqref{eq:Gauss-curvature-equation} constitutes an \emph{elliptic} Monge--Amp\`ere equation.  In the present case where $K_{\mathcal S} < 0$, the natural analog  is to consider solutions of Lorentzian signature  $(-, +, \dots, +)$, so that \eqref{eq:Gauss-curvature-equation} becomes a \emph{hyperbolic} Monge--Amp\`ere equation.

A general strategy to obtain solutions to (fully) non-linear problems like \eqref{eq:Gauss-curvature-equation} is to first linearize the equation around a fixed solution and then to show sufficiently strong estimates for the linearized equation. To eventually solve the nonlinear problem, a version of the implicit function theorem (for carefully chosen Banach spaces) will then show the existence of a solution in a neighborhood of the initial solution.  In addition, if sufficiently robust estimates are proved for the linearized problem, one may actually prove the existence of solutions far away from the initial solution by following a suitable path in the solution space. Moreover, other properties such as the regularity of solutions can be shown in this framework. We refer to \cite[Chapter 3.1.4.]{MR3617963} for more details about this method in the elliptic case.

The purpose of this paper is to investigate this strategy in the context of the hyperbolic Monge--Amp\`ere equation \eqref{eq:Gauss-curvature-equation} with $K_{\mathcal S}<0$.

\section{The linearized  prescribed negative Gauss curvature problem}
 In the following we will  derive the linearization of the prescribed Gauss curvature equation \eqref{eq:Gauss-curvature-equation} around a hypersurface $\mathcal S$ given as a graph over $u_{\mathcal S}$ with Lorentzian Hessian. To this end, given a $C^2$ function $u$, we introduce the operator
 \begin{align*} 
\Psi(u) :=  \frac{\det D^2 u }{(1+|Du|^2)^{\frac{n+2}{2}}} .
 \end{align*}

\begin{proposition}
	Denote the linearized operator associated to $\Psi$ around $u_{\mathcal S}$ with $L_{u_{\mathcal S}}$. Then,
\begin{align} 
L_{u_{\mathcal S}} (v) = \Psi(u_{\mathcal S}) \left( \operatorname{tr} \left(  (D^2u_{\mathcal S})^{-1}  D^2 v \right) -   (n+2) \frac{Du_{\mathcal S}\cdot  Dv}{ 1+|Du_{\mathcal S}|^2} \right).  \label{eq:linearized-equation}
\end{align}

\begin{proof}
Consider $u_\epsilon:= u_{\mathcal S} +\epsilon v$. Then, taking the derivative of $\Psi(u_\epsilon)$  with respect to $\epsilon$  and using the Jacobi formula for the derivative of the determinant (see for instance \cite[Lemma A.1]{MR3617963}), we obtain  \begin{align*}
\frac{\d}{\d \epsilon}\Psi(u_\epsilon)   = &  \frac{\det(D^2 u_\epsilon)  \operatorname{tr} \left(  (D^2u_\epsilon)^{-1}  D^2 v \right) }{(1+|D u_\epsilon |^2)^{\frac{n+2}{2}}} -   (n+2) \frac{\det(D^2 u_\epsilon) Du_\epsilon\cdot  D v }{ (1+|D u_\epsilon  |^2)^{\frac{n+4}{2}}}   \\ 
=   &\Psi(u_\epsilon) \left(\operatorname{tr} \left(  (D^2u_\epsilon)^{-1}  D^2 v \right) - (n+2) \frac{Du_\epsilon\cdot  Dv}{ 1+|Du_\epsilon|^2} \right).
\end{align*}
The result follows by taking $\epsilon=0$.
\end{proof}
\end{proposition}

By considering the operator \eqref{eq:linearized-equation} we naturally ask what type of PDE it constitutes. It turns out that, for dimensions $n\geq 3$, we can realize \eqref{eq:linearized-equation} as a Laplace--Beltrami operator for a suitable pseudo--Riemannian metric that is conformal to the Hessian of $u_{\mathcal S}$.  
\begin{proposition}\label{prop:equation-as-geometric-wave-equation}
	 Let $n\geq 3$. 
The linearized prescribed Gauss curvature operator $L_{u_{\mathcal S}}$ in \eqref{eq:linearized-equation} is, up to the global factor \begin{align}\label{eq:definition-of-f}
f = -  | \det D^2u_{\mathcal S}|^{- \frac{n-1}{n-2}} ( 1+ |Du_{\mathcal S}|^2)^{ \frac{n(n+2)}{2(n-2)}},
\end{align}
the Laplace--Beltrami operator for the Lorentzian metric 
\begin{align}\label{eq:Lorentzian-metric}
g = g_{u_{\mathcal S}} = |\det(D^2 u_{\mathcal S})|^{\frac{1}{n-2}} (1+|D u_{\mathcal S}|^2 )^{-\frac{n+2}{n-2}} D^2u_{\mathcal S}. 
\end{align}
More precisely, the linearized Gauss curvature operator \eqref{eq:linearized-equation} satisfies
 \begin{align*}
\Box_g v = f L_{u_{\mathcal S}} v.
\end{align*}

 For $n=2$, the linearized prescribed Gauss curvature operator $L_{u_{\mathcal S}}$ can be written as
\begin{align}\label{eq:linearized-equation-n=2}
	L_{u_{\mathcal S}} = \Psi(u_{\mathcal S}) \left( \Box_m +  b^j \partial_j\right)
\end{align}
for  the Lorentzian metric $m=m_{u_{\mathcal S}}=D^2 u_{\mathcal S}$, 
with \begin{align*}b^j = -4\frac{\delta^{ji } \partial_i u_{\mathcal S}}{1+|Du_{\mathcal S}|^2}  + \frac{1}{2} m^{ji} \operatorname{tr} (m^{-1} \partial_i m ),\end{align*}
where we use the Einstein summation convention, and $m^{ij}$ denotes the inverse of $m_{ij}$. 
\begin{remark}
In the elliptic case this shows that, for dimensions $n\geq 3$, the linearized operator can be realized as the Laplacian on the analogous Riemannian manifold with metric \eqref{eq:Lorentzian-metric}. 
\end{remark}
\begin{remark}
For $n=2$, we see that we cannot write $L_{u_{\mathcal S}}$ as a Laplace--Beltrami operator for a Lorentzian metric.
\end{remark}
\begin{proof} We first let $n\geq 3$ and define $m := D^2 u_{\mathcal S}$. Then, we write  \eqref{eq:linearized-equation} as 
\begin{align}\label{eq:waveequationforgeneralperturbation}
(\Psi(u_{\mathcal S}) )^{-1} L_{u_{\mathcal S}}(v) =  	m^{ij} \partial_i \partial_j v - (n+2) \frac{\delta^{ij} \partial_i u_{\mathcal S} \partial_j v}{1 + |Du_{\mathcal S}|^2}.
\end{align}
We recall that the Laplace--Beltrami operator associated to a metric $g_{ij}$ is given as
\begin{align*}
	\Box_g = g^{ij} \partial_{i} \partial_j v + \frac{1}{\sqrt{|\det g |}} \partial_i ( g^{ij} \sqrt{|\det g|} ) \partial_j v 
\end{align*}
and we want to realize the right-hand side of \eqref{eq:waveequationforgeneralperturbation} as such an operator, up to a global factor. For the principal part, we obtain  $g^{ij} = \omega m^{ij}$ for some   factor $\omega$. Thus, $\det(g) = \omega^{-n}\det(m) $ and $\omega$ has to obey
\begin{align*}
	- (n+2) \frac{\partial_j u_{\mathcal S}}{1 + |Du_{\mathcal S}|^2}  \omega  & = \frac{\omega^{n/2} }{ \sqrt{|\det m|} } \partial_i \left(\omega^{1-n/2}  m^{ij} \sqrt{ |\det m| }\right), 
\end{align*}
or 
\begin{align}
\label{eq:grad us}
	-(n+2)  \frac{\textup{grad}  (u_{\mathcal S})}{1 + |Du_{\mathcal S}|^2} &= \frac{\omega^{n/2-1}}{\sqrt{|\det m|}} \textup{div} \left(m^{-1}  \omega^{1-n/2} \sqrt{|\det m| } \right) .
\end{align}
Now, using that $m^{-1} \det m$ is divergence free (being equal to the cofactor matrix of $m=D^2 u_{\mathcal S}$), the right hand side simplifies to 
\begin{align*}
	m^{-1}  \operatorname{grad} \log\left( \omega^{1-n/2} |\det m|^{-\frac 12}\right).
\end{align*}
Since the left hand side of \eqref{eq:grad us} can also be re-written as 
\begin{align*}
	-\frac{(n+2)}{2} m^{-1} \operatorname{grad}\log ( 1+|Du_{\mathcal S}|^2 )=m^{-1} \operatorname{grad}\log ( 1+|Du_{\mathcal S}|^2 )^{-\frac{n+2}{2}},
\end{align*}
we conclude that
\begin{align*}
	\omega^{1-n/2} := \sqrt{|\det m|} \left(1+|Du_{\mathcal S}|^2\right) ^{-\frac{n+2} {2}}
\end{align*}
or equivalently
\begin{align*}
	\omega  := | \det m|^{\frac{1}{2-n}} ( 1+ |Du_{\mathcal S}|^2)^{\frac{n+2}{n-2}}.
\end{align*}
Hence, with this choice, we finally obtain 
\begin{align*} 
	g_{ij} = | \det D^2 u_{\mathcal S}|^{\frac{1}{n-2}} (1+ |Du_{\mathcal S}|^2 )^{-\frac{n+2}{n-2} } ( D^2u_{\mathcal S})_{ij}
\end{align*}
with \begin{align*}\det g= \det(D^2u_{\mathcal S}) | \det D^2 u_{\mathcal S}|^{\frac{n}{n-2}} \left(1+|Du_{\mathcal S}|^2\right)^{-n (n+2)/(n-2)}.  \end{align*}
Thus, $(\Psi(u_s) )^{-1} L_{u_{\mathcal S}} = \omega^{-1} \Box_g$ and $f = \frac{\omega}{\Psi(u_s)}$. 

For the case $n=2$, we immediately see that the above computation fails and indeed inspecting \eqref{eq:waveequationforgeneralperturbation} we directly obtain \eqref{eq:linearized-equation-n=2}. 
\end{proof}
\end{proposition}

\section{The nonlinear prescribed negative Gauss curvature problem: Local analysis}
\label{sec:local-analysis}

In the previous section we have seen that, given a hypersurface $\mathcal S={\rm graph}(u_{\mathcal S})$ with Lorentzian Hessian $D^2u_{\mathcal S}$, the linearization  for the prescribed negative Gauss curvature problem around $\mathcal S$ constitutes a hyperbolic PDE. We will now consider the problem of constructing nearby hypersurfaces arising from small, localized perturbations of the Gauss curvature. 
Due to the hyperbolic nature of the equations, we will assume certain properties on the causal structure of the domain $\Omega\subset \mathbb R^n$ with respect to Lorentzian metric $m_{u_\mathcal S}:=D^2u_{\mathcal S}$ or, equivalently, its conformal rescaling $g_{u_\mathcal S}$ (see \eqref{eq:Lorentzian-metric})\footnote{In the case of $n=2$, we replace $g_{u_{\mathcal S}}$ by $m_{u_{\mathcal S}}=D^2u_{\mathcal S}$, see \eqref{eq:linearized-equation-n=2}.}.

We assume that $\Omega\subset \mathbb R^n$ is  a compact  lens-like shaped domain with boundary associated to  a smooth spacelike\footnote{A hypersurface $\Sigma\subset \mathbb R^n$ is called ``spacelike'' if  $g_{u_{\mathcal S}}\vert_{\Sigma}$ is Riemannian.} compact oriented hypersurface $\Sigma_0\subset \mathbb R^n $ with boundary. More precisely, we assume that the domain $\Omega$ is the image of a smooth diffeomorphism   $\varphi\colon [0,1]\times \Sigma_0\to \Omega \subset  \mathbb R^n $, where we require that:\\
 (1) for each $t\in [0,1]$, $\Sigma_t:= \varphi(t, \Sigma_0)$ is a   spacelike hypersurface;\\
  (2) $\varphi(0,\cdot)\colon \Sigma_0 \to \Sigma_0$ is the identity;\\
   (3) $\Sigma_0$ is a Cauchy hypersurface for $\Omega$;\footnote{Given a curve $\gamma:\mathbb R\to \mathbb R^n$, $\gamma$ is said ``timelike'' if $g_{u_{\mathcal S}}(\dot\gamma(t),\dot\gamma(t))<0$ for every $t \in \mathbb R$. Also, $\gamma$ is said ``inextendible timelike'' if there is no timelike curve that strictly contains it.
   Finally, given a domain $\mathcal U\subset \mathbb R^n$, a hypersurface $\Sigma\subset \mathcal U$ is called ``Cauchy for $\mathcal U$'' if, inside $\mathcal U$, it is met exactly once by every inextendible timelike curve.} \\
  (4) $ \varphi ([0,1] \times \partial \Sigma_0 )$ is spacelike.\\
   Since $\varphi  $ is a diffeomorphism, by a mild abuse of notation,
   we can define a smooth ``time'' map $t \colon \Omega \to [0,1]$ saying that $t(p)=t$ if $p \in \Sigma_t$ (hence, $p \in \Sigma_{t(p)}$ for any $p \in \Omega$).  Since the leaves $\Sigma_t$ are spacelike, we note  that $t$ is a smooth temporal function  satisfying $-g_{u_\mathcal S}^{-1} (\d t, \d t)>0$, which is uniformly bounded from above and uniformly bounded away from zero. We time-orient $\Omega$ by imposing that $-\nabla^{g_{\mathcal S}} t$ is future directed.  We shall use the notation $\Xi_t := \varphi( [0,t]\times \partial \Sigma_0)$, so that the boundary of $\Omega$ can be written as $\partial \Omega = \Sigma_0 \cup \Sigma_1 \cup \Xi_1$, where we recall that each piece is spacelike and that $\Sigma_0$ is a Cauchy hypersurface for $\Omega$.

Note that, for small perturbations of $u_{\mathcal S}$, i.e., $m_u := D^2u$ with $\|u-u_{\mathcal S}\|_{C^2(\Omega)}$ sufficiently small, the causal structure of $\Omega$ remains stable and $\Sigma_t$ will still foliate $\Omega$ as spacelike hypersurfaces with respect to the causal structure given by $m_u$. In particular, $\Omega$ is globally hyperbolic\footnote{By definition, a manifold is ``globally hyperbolic'' if it contains a Cauchy hypersurface. There are several other equivalent definition of global hyperbolicity, relating the existence of a  Cauchy hypersurface to the causal structure of $\Omega$.} with $\Sigma_0$ as Cauchy hypersurface.

To apply this construction to our setting, given a smooth function $u_{\mathcal S}:\mathbb R^n \to \mathbb R$ with Lorentzian Hessian $D^2u_{\mathcal S}$ and fixed a point $p \in \mathbb R^n$, it is always possible to construct a domain $\Omega$ containing $p$ satisfying the properties above. We have illustrated the domain $\Omega$ in \cref{fig:omega}.

 \begin{figure}
 	\centering
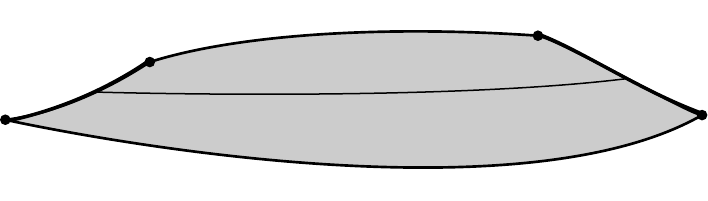
\caption{Illustration of the  domain $\Omega$, the leaves  $\Sigma_t$ and the (possibly disconnected) boundary component $\Xi_1$.}
\label{fig:omega}
 \end{figure}

\subsection{Local solvability: \texorpdfstring{\cref{thm:local-existence}}{Theorem 1}}
We will now state the local existence theorem, where we work in the standard Fréchet space topology for $C^\infty(\Omega)$. 
\begin{theorem}
	\label{thm:local-existence}
For all $\eta \in C^\infty(\Omega)$ in a neighborhood of zero, there exists a solution $u_\eta \in  C^\infty(\Omega)$ satisfying 
	\begin{align}\label{eq:prescribed-curvature-equation}
		\frac{\det(D^2 u_\eta)}{(1+|Du_\eta|^2)^{\frac{n+2}{2}}} = K_{\mathcal S} +   \eta
	\end{align}
	 with $u_\eta\vert_{\Sigma_0}=u_{\mathcal S}\vert_{\Sigma_0}$ and $D u_\eta\vert_{\Sigma_0} =D u_{\mathcal S} \vert_{\Sigma_0}$. Moreover, the solution map
	 $$
	 C^\infty(\Omega) \ni \eta \mapsto u_\eta \in  C^\infty(\Omega)
	 $$
	 is smooth. 
 \end{theorem} 

\begin{remark}
The choice of the boundary conditions in \cref{thm:local-existence} can be relaxed. First, we could have chosen other  data for $u_\eta$ and $D u_\eta$ on $\Sigma_0$ as long as they are sufficiently close to $u_{\mathcal S}$ and $D u_{\mathcal S}$.  Secondly, instead of $\Sigma_0$, we could have also chosen any other Cauchy hypersurface and pose appropriate data there. 
\end{remark}

\begin{remark}
The choice of domain $\Omega$ can also be relaxed. For example, $\Omega$ does not necessarily have to be globally hyperbolic. In particular, $\Omega$ could also have a piece of timelike boundary\footnote{We recall that a hypersurface is spacelike/timelike if the induced metric on the hypersurface is Riemannian/Lorentzian, respectively.} and, in addition to Cauchy data on a suitable spacelike hypersurface, one would  have to prescribe  Dirichlet data on the timelike piece of the boundary. 
\end{remark}

\begin{remark}
In the elliptic case, an analogous result can be shown for Dirichlet boundary conditions using the implicit function theorem (see, for instance, \cite[Chapter 3.1.4.]{MR3617963}). However:\\
(i) In the hyperbolic case, well-posedness is only expected for the present Cauchy problem and generically we cannot achieve that $u_\eta = u_{\mathcal S}$ on the whole boundary $\partial \Omega$.\\
(ii) Hyperbolic problems (in contrast to elliptic ones) suffer from a loss of derivatives, so one cannot use the implicit function theorem.
Still, this loss of derivatives can be salvaged by using the more elaborate Nash--Moser iteration scheme, which we will use to prove \cref{thm:local-existence}. 
\end{remark}

\begin{remark}
Another possible way to show \cref{thm:local-existence} would be to differentiate \eqref{eq:prescribed-curvature-equation} and write it as a system of quasilinear wave equations, so as to avoid the use of Nash--Moser iteration. In particular, setting $v:= Du$ and differentiating \eqref{eq:prescribed-curvature-equation}, for $n\geq 3$ we obtain 
\begin{align*}
L_{v} v =  D(K_s + \eta),
\end{align*}
where $L_v$ is the linearized operator of \eqref{eq:linearized-equation} with $Du_{\mathcal S}$ replaced by $v$. 
In this case, one would also have to  propagate integrability conditions on $v$, so to find a function $u$ satisfying $v = Du$. 
\end{remark}

\begin{remark}
	We have stated \cref{thm:local-existence} in the $C^\infty$ category. By tracking down the exact regularity exponents in the proof,  it should be possible to show a quantitative version of \cref{thm:local-existence} in lower regularity. In order to show a result with optimal regularity, however, it would be probably beneficial to avoid Nash--Moser iteration but rather to rewrite \eqref{eq:prescribed-curvature-equation} as a quasilinear system for which one can derive sharper estimates. 
	\end{remark}

The key ingredient for the use of the Nash--Moser iteration (in fact, also the key ingredient in the case of a proof  via a quasilinear system) is a robust energy estimate for the linearized equation, which will be the content of the next lemma. To do so we first introduce the future directed timelike normal to the leaves of the foliation $$N_{\mathcal S}:= - \frac{\nabla^{g_{u_{\mathcal S}}} t}{|\nabla^{g_{u_{\mathcal S}}}t|_{g_{u_{\mathcal S}}} }$$ with respect to the Lorentzian metric induced by $g_{u_{\mathcal S}}$ (see \eqref{eq:Lorentzian-metric}), where we use the notation $ |\nabla^{g_{u_{\mathcal S}}}t|_{g_{u_{\mathcal S}}}  = (-g(\nabla^{g_{u_{\mathcal S}}}t,\nabla^{g_{u_{\mathcal S}}}t))^{\frac 12}$.  In the case of $n=2$, we replace $g_{u_{\mathcal S}}$ by $m_{u_{\mathcal S}}=D^2u_{\mathcal S}$. Note that $N_{\mathcal S}$  will remain timelike  with respect to all metrics $g_u$ for which $\|u-u_{\mathcal S}\|_{C^2(\Omega)}$ is sufficiently small. 
In particular, for all $u\in C^2(\Omega)$ with $\| u-u_{\mathcal S} \|_{C^2(\Omega)}<\epsilon_0$, the following constitutes a well-posed initial value  problem
\begin{align}\label{eq:linearized-equation-nash-moser}
	\begin{cases}
		L_{u}(v) = F\text{ in } \Omega \\
		v|_{\Sigma_0} = v_1\\
	N_{\mathcal S}v|_{\Sigma_0 } =v_2,
	\end{cases}
\end{align}
where  $L_u$ is the linearized operator from \eqref{eq:linearized-equation}, $v_1 , v_2 \in C^\infty(\Sigma_0)$, and $F\in C^\infty(\Omega)$. 
We will use the standard Euclidean background metric (and induced volume form) for the definitions of the Sobolev spaces $H^k(\Omega)$ and $H^{k}(\Sigma_t)$ for $k\geq 0$ on the compact manifolds $\Omega$ and $\Sigma_t$.  We will also use the notation $|\partial v |^2 = \sum_{i =1}^n |\partial_{x_i} v |^2$.  

\subsection{An energy estimate for the linearized problem}

\begin{proposition}[Energy estimate]\label{lem:energy-estimate}
	Let $u\in C^\infty(\Omega)$ satisfy $\| u-u_{\mathcal S} \|_{H^{[\frac{n-1}{2}] + 6} (\Omega)}\leq \epsilon_0$ for some $\epsilon_0 >0$ sufficiently small depending on $\Omega, u_{\mathcal S},n$.  Let $v\in C^\infty(\Omega)$ be the unique smooth solution to \eqref{eq:linearized-equation-nash-moser}. Then $v$ satisfies the following energy estimate:
\begin{multline} \nonumber
		\| v\|_{H^{k}(\Omega)}  \lesssim_{u_{\mathcal S}, \Omega,k,n}\;\;  \Big(\| F\|_{H^{[\frac n2]+1}(\Omega)} +  \|  v_1\|_{H^{[\frac{n-1}{2}]+3}(\Sigma_0)} +\|  v_2 \|_{H^{[\frac {n-1}{2}]+2}(\Sigma_0)}\Big) \| u\|_{H^{k+2}(\Omega)}\\
		+\|F \|_{H^{k-1}(\Omega)}   +    \|  v_1\|_{H^{k}(\Sigma_0)}  +  \|  v_2\|_{H^{k-1}(\Sigma_0)}  
	\end{multline}
for all $k\geq [\frac{n-1}{2}] + 3$. 
	\begin{proof}
		Throughout the proof we assume that implicit constants appearing in $\lesssim$ may depend on $\Omega, u_{\mathcal S}, n,\epsilon_0, k$. 
		
	We first consider the case $n\geq 3$. In this case, \eqref{eq:linearized-equation-nash-moser} is equivalent to  
		\begin{align} \label{eq:geometric-wave-equation}
			\Box_{g} v = G
		\end{align} 
		where $$
		g= |\det(D^2 u)|^{\frac{1}{n-2}} (1+|D u|^2 )^{-\frac{n+2}{n-2}} D^2u\qquad \text{and} \qquad G = f F,
		$$
		 with $f$ depending on $Du$ and $D^2u$ as defined in \eqref{eq:definition-of-f}.

		 We recall the standard energy identity  (i.e.,  the divergence theorem) for a  vector field $X$ on $\Omega$ for \eqref{eq:geometric-wave-equation}:
		\begin{align}\nonumber 
			\int_{\Sigma_1} J_\mu^X[v] N_{\Sigma_1}^\mu  \,\textup{vol}^g_{\Sigma_1}+ 	\int_{\Xi_1} & J_\mu^X[v] N^\mu_{\Xi_1}\,\textup{vol}^g_{\Xi_1} + \int_\Omega K^X[v]\,\textup{vol}^g_{\Omega} \\
			& = \int_{\Sigma_0} J_\mu^X[v] N_{\Sigma_0}^\mu  \,\textup{vol}^g_{\Sigma_0}+ \int_{\Omega} G\, X(v) \,\textup{vol}^g_{\Omega},\label{eq:geometric-energy-estimate}
		\end{align}
		where $K^X[v]= \frac 12  (\mathcal L_X g)^{\mu \nu} T_{\mu \nu}[v]$, $J^X[v]= T(X, \cdot )[v]$, and $N_{\Sigma_0}, N_{\Sigma_1} , N_{\Xi_1}$ are the future directed normal vector fields to $\Sigma_0, \Sigma_1$ and $\Xi_1$, respectively. Note that   $N := N_{\Sigma_t}=- \frac{\nabla^{g} t}{|\nabla^{g} t|_g }$ for each $t\in [0,1]$.  Further, $T[v]=(T_{\mu\nu}[v])$ and $T_{\mu\nu}[v] = \partial_\mu v \partial_\nu v - \frac 12 g_{\mu\nu} g^{\alpha \beta }\partial_\alpha v \partial_\beta v$ is the energy momentum tensor associated to \eqref{eq:geometric-wave-equation}.
		
		We will now apply the energy identity with the  timelike vector field 
		\begin{align*}X = e^{-a t} N 
		\end{align*} for some $a>0$ sufficiently large. 

First, we remark that for the future timelike vector field $X$  and any other future directed timelike vector field $Y$  we  have   $J_\mu^X[v] Y^\mu   \sim e^{-a t} |\partial v|^2 \geq 0 $ as $\Omega$ is compact and  $\| u-u_{\mathcal S}\|_{H^{k_0}} \leq \epsilon_0$. Hence, 
$$
J^X_\mu[v] N_{\Sigma_t}^\mu \sim  e^{-a t} |\partial v|^2\qquad\text{ and }\qquad J^X_\mu[v] N_{\Xi_t}^\mu \sim  e^{-a t} |\partial v|^2\qquad \forall\,t\in [0,1].
$$
In particular $\int_{\Xi_1}   J_\mu^X[v] N^\mu_{\Xi_1}\,\textup{vol}^g_{\Xi_1}  \geq 0$, so we can throw away this term from our energy estimate. 
	
	For the spacetime term involving $K^X$ in \eqref{eq:geometric-energy-estimate}, we note that
		\begin{align*}K^X =\frac 12 ( \mathcal L_{X} g)^{\mu \nu} T_{\mu \nu} =\frac 12 e^{-a t} (\mathcal L_N g)^{\mu \nu} T_{\mu \nu} + a e^{-a t} T(-\nabla^g t, N) ,
		\end{align*} 
where
	$$
	  T(-\nabla^g t, N)   \sim |\partial v|^2\qquad \text{ and }
	  \quad |(\mathcal L_{N}g)^{\mu\nu} T_{\mu \nu} |\lesssim |\partial v|^2.
	  $$
	Thus, by the co-area formula, we obtain
	\begin{align*}
 \int_\Omega & K^X[v] \textup{vol}^g_{\Omega}  = \int_0^1 \int_{\Sigma_t} K^X[v] \frac{1}{|\nabla^{g} t|_g}\,\textup{vol}^g_{\Sigma_t} \d t  \\ &  \sim \int_0^1   \int_{\Sigma_t} \left( \frac 12 (\mathcal L_{N}g)^{\mu\nu} T_{\mu \nu}  + a T(- \nabla^g t, N) \right) e^{- at} \,\textup{vol}^g_{\Sigma_t} \d t \sim   a \int_0^1 \int_{\Sigma_t}  e^{-a t} |\partial v|^2 \,\textup{vol}^g_{\Sigma_t} \d t 
\end{align*} for all $a>0$ sufficiently large.

	For the term   $\int_{\Omega} G\, X(v) \,\textup{vol}_{\Omega}$ we use Young's inequality
	$$
	\int_{\Omega} G\, X(v) \,\textup{vol}_{\Omega}=\int_{\Omega}e^{-at} G\, N(v) \,\textup{vol}_{\Omega} \lesssim \int_{\Omega} e^{-at} G^2   \,\textup{vol}_{\Omega}+\int_{\Omega} e^{-at}|\partial v|^2 \,\textup{vol}_{\Omega}.
	$$
In this way, we end up with the bound 
\begin{align*}
\int_{\Sigma_1} e^{- at } |\partial v|^2 + a \int_0^1 \int_{\Sigma_t} e^{-a t} |\partial v|^2 & \lesssim  \int_{\Omega} e^{- at}   |G|^2 + 	\int_{\Sigma_0}   |\partial v|^2 
\end{align*}
(here we integrate with respect to the Euclidean volume form induced on $\Sigma_t$ and $\Omega$, respectively, and we use that Euclidean volume forms on $\Omega$ and $\Sigma_t$ are comparable to the volume form induced by $g$). 
Actually, by integrating only in the region from $\Sigma_0$ to $\Sigma_t$ for $t\leq 1$, and then taking the supremum in $t$, we obtain the slightly better estimate:
\begin{align} \label{eq:energy-estimate-with-weight}
\sup_{0\leq t \leq 1}	\int_{\Sigma_t} e^{- at } |\partial v|^2 + a \int_0^1 \int_{\Sigma_t} e^{-a t} |\partial v|^2 \lesssim  \int_{\Omega} e^{- at}   |G|^2+	\int_{\Sigma_0}  |\partial v|^2
\end{align}
	for all $a>0$ sufficiently large depending on $\Omega, u_{\mathcal S}, n$.

For higher order estimates, we will commute the equation with a collection of fixed  vector fields $(Z_i)_{1\leq i \leq n}$, defined on $\Omega $, which are tangential to $\Sigma_t$ for each $t$ and span the tangent space to $\Sigma_t $ at each point. For concreteness we fix a collection of vector field spanning the tangent bundle of $\Sigma_0$ and use $\varphi(t, \cdot)\colon \Sigma_0 \to \Sigma_t$ to pushforward them to vector fields on $\Omega$, which are by construction tangential to each $\Sigma_t$.   We will use standard multiindex notation $Z^\alpha$ and remark that $\| f \|_{H^k(\Sigma_t)} \sim \sum_{|\alpha | \leq k} \|  Z^{\alpha} f \|_{L^2(\Sigma_t)}$, where the implicit constant depends on the choice of vector fields, which we keep fixed from now on. Note that
		 \begin{align*}\Box_g Z^{\alpha} v = Z^{\alpha} G - [\Box_g , Z^{\alpha}] v.\end{align*} 
		Hence, from \eqref{eq:energy-estimate-with-weight}  we get
		\begin{align*}\nonumber
		\sup_{0\leq t \leq 1}	\int_{\Sigma_t} e^{- at } & |\partial Z^\alpha v|^2 +	a \int_{0}^1 e^{-a t} \|\partial Z^{\alpha} v\|_{L^2(\Sigma_t)}^2\d t\\& \lesssim \int_{0}^1 e^{- at} \Big( \|Z^{\alpha} G\|_{L^2(\Sigma_t)}^2 + \| [\Box_g, Z^{\alpha}] v\|_{L^2(\Sigma_t)}^2\Big) \d t + 	\int_{\Sigma_0} |\partial Z^\alpha v|^2 .
		\end{align*}
	We now sum over all multi-indices  $|\alpha|\leq k-1$ for $k\geq [\frac{n-1}{2}]+3$, for which we use the notation $Z^{\leq k-1}$. 
	At the top order in derivatives on $v$ the commutator $ [\Box_g, Z^{\alpha}] $ vanishes. Hence, from standard interpolation estimates and using also the definition of $g$ in terms of $Du$ and $D^2u$, we have 
		\begin{multline} \label{eq:commutation-estimate} \| [\Box_g, Z^{\leq k-1}] v\|_{L^2(\Sigma_t)}^2  
			 \lesssim    \|  \partial v\|_{H^{k-1}(\Sigma_t)}^2 + \| G\|^2_{H^{k-1}(\Sigma_t)}\\+
			  \Big(\|Z^{\leq 1}\partial v\|_{L^\infty(\Sigma_t)}^2 +\| G\|^2_{L^\infty(\Sigma_t)} \Big)\|  \partial  D^2 u\|_{H^{k-1}(\Sigma_t)}^2.
		\end{multline} 
Here we used our assumption $u \in H^{[\frac{n-1}{2}]+6}(\Omega) \hookrightarrow C^3(\Omega)$, so that $D^2u \in C^1(\Omega)$ and $(D^2u)^{-1} \in C^1(\Omega)$, and thus $g^{-1}, g \in C^1(\Omega)$. 
The inhomogeneous term $G$ appears from using \eqref{eq:geometric-wave-equation} to control two derivatives normal to $\Sigma_t$   with tangential derivatives and $G$. Note that $\Box_g = -N^2 + \Delta_{\Sigma_t} + P_1$, where $P_1$ is a first order operator.

Similarly, we have 
\begin{align*}
\|  \partial Z^{\leq k-1} v \|_{L^2(\Sigma_0)} \lesssim \|  N_{\mathcal S} v\|_{H^{k-1} (\Sigma_0)} + \| v\|_{H^{k}(\Sigma_0)} \lesssim\|  v_2 \|_{H^{k-1}(\Sigma_0)} + \| v_1 \|_{H^k(\Sigma_0)}. 
\end{align*}

Thus,
\begin{align*} \nonumber 
	\sup_{0\leq t \leq 1}	& e^{- at } \|\partial v\|^2_{H^{k-1}(\Sigma_t)}  + a \int_0^1 e^{-at } \|\partial v\|^2_{H^{k-1}(\Sigma_t)} \d t  \lesssim \int_0^1 e^{-at} \|  G\|_{H^{k-1}(\Sigma_t) }^2  \d t\\
	& + \int_0^1 e^{-a t}   \| \partial v\|_{H^{k-1}(\Sigma_t)}^2  \d t\nonumber  +  \int_0^1 e^{-a t} \Big(\|  \partial v \|^2_{H^{[\frac {n-1}{2}]+2}(\Sigma_t)} + \|G\|_{H^{[\frac n2]+1} (\Sigma_t)}\Big)  \|\partial D^2 u\|_{H^{k-1}(\Sigma_t)}^2  \d t    \\ & +\|  v_2 \|_{H^{k-1}(\Sigma_0)}^2 + \| v_1 \|_{H^k(\Sigma_0)}^2 .
\end{align*}

Now we  note that, for all $a>0$ sufficiently large (also depending on $k$), we can absorb the second term on the right hand side into the left to obtain
\begin{align} \nonumber 
	\sup_{0\leq t \leq 1}	& e^{- at } \|\partial v\|^2_{H^{k-1}(\Sigma_t)}  + a \int_0^1 e^{-at } \|\partial v\|^2_{H^{k-1}(\Sigma_t)} \d t  \lesssim \int_0^1 e^{-at} \|  G\|_{H^{k-1}(\Sigma_t) }^2  \d t\\
	&    +  \int_0^1 e^{-a t} \Big(\|  \partial v \|^2_{H^{[\frac{n-1}{2}]+2}(\Sigma_t)}+ \|G\|_{H^{[\frac n2]+1} (\Sigma_t)}\Big)   \|\partial D^2 u\|_{H^{k-1}(\Sigma_t)}^2  \d t\label{eq:higher-order-energy-estimate}\\
	& + \| v_1 \|_{H^k(\Sigma_0)}^2 +\|  v_2 \|_{H^{k-1}(\Sigma_0)}^2 \nonumber  .
\end{align}

Now, we first use \eqref{eq:higher-order-energy-estimate} with $k=[\frac{n-1}{2}]+3$, and in the right hand side we use the trace estimate
 \begin{align*}
 	\sup_{t\in[0,1]} \|\partial D^2 u\|_{H^{ [\frac{n-1}{2}]+2}(\Sigma_t)}^2 \lesssim \|u \|^2_{H^{[\frac{n-1}{2}] +6}(\Omega)} \lesssim 1
\end{align*}
(recall that $\| u-u_{\mathcal S} \|_{H^{[\frac{n-1}{2}] + 6} (\Omega)}\leq \epsilon_0$).
In this way, choosing $a>0$ sufficiently large, we can absorb the second term on the right hand side into the left one, and we obtain
\begin{align*}
 	\sup_{0\leq t \leq 1}	& \|\partial v\|^2_{H^{[\frac{n-1}{2}]+2}(\Sigma_t)} 
 	 \lesssim   \|G\|_{H^{[\frac n2]+1}(\Omega)}^2 +  \|  v_1\|^2_{H^{[\frac {n-1}{2}]+3}(\Sigma_0)} +\|  v_2 \|_{H^{[\frac{n-1}{2}]+2}(\Sigma_0)}^2 .
\end{align*}
Plugging this back into \eqref{eq:higher-order-energy-estimate} yields
\begin{align*}\nonumber
\int_0^1 \| \partial v\|_{H^{k-1}(\Sigma_t)}^2 \d t\lesssim & \| G\|_{H^{k-1}(\Omega)}^2 \\
&+\Big( \|G\|_{H^{[\frac{n}{2}]+1}(\Omega)}^2 +  \|  v_1\|^2_{H^{[\frac{n-1}{2}]+3}(\Sigma_0)} +\|  v_2 \|_{H^{[\frac{n-1}{2}]+2}(\Sigma_0)}^2 \Big)  \| u\|_{H^{k+2}(\Omega)}^2 \\ & +   \|  v_1\|^2_{H^{k}(\Sigma_0)}  +  \|  v_2\|^2_{H^{k-1}(\Sigma_0)}  
\end{align*}
for all $k\geq [\frac{n-1}{2}]+3$. 

In order to control higher order derivatives transversal to $\Sigma_t$, we use again that $v$ solves   \eqref{eq:geometric-wave-equation}. Then, arguing similarly as before using  interpolation estimates, we obtain
\begin{align*}\nonumber
\|  v\|_{H^{k}(\Omega)}^2 \lesssim & \| G\|_{H^{k-1}(\Omega)}^2 +\Big( \|G\|_{H^{[\frac n2]+1}(\Omega)}^2 +  \|  v_1\|^2_{H^{[\frac{n-1}{2}]+3}(\Sigma_0)} +\|  v_2 \|_{H^{[\frac{n-1}{2}]+2}(\Sigma_0)}^2 \Big)  \| u\|_{H^{k+2}(\Omega)}^2 \\ & +   \|  v_1\|^2_{H^{k}(\Sigma_0)}  +  \|  v_2\|^2_{H^{k-1}(\Sigma_0)}  .
\end{align*}
 
Moreover, for $k\geq  [ \frac{n}{2}]+1$,  
\begin{align*}
	\|G\|_{H^k} = \| f F\|_{H^k} &\lesssim  \| f \|_{H^k} \|F\|_{L^\infty} + \|   F\|_{H^k} \| f\|_{L^\infty}\\ & \lesssim \| u\|_{H^{k+2}} \| F\|_{H^{[\frac n2]+1}} + \| u\|_{H^{[\frac n2]+3}} \| F\|_{H^{k} } \lesssim  \| u\|_{H^{k+2}} \| F\|_{H^{[\frac n2]+1}}  + \| F\|_{H^{k} }.
\end{align*}
Thus,
\begin{align*}\nonumber
	\| v\|_{H^{k}(\Omega)}  \lesssim & \|F \|_{H^{k-1}(\Omega)}   + \Big(\| F\|_{H^{[\frac n2]+1}(\Omega)} +  \|  v_1\|_{H^{[\frac{n-1}{2}]+3}(\Sigma_0)} +\|  v_2 \|_{H^{[\frac{n-1}{2}]+2}(\Sigma_0)}\Big) \| u\|_{H^{k+2}(\Omega)}\\
	& +   \|  v_1\|_{H^{k}(\Sigma_0)}  +  \|  v_2\|_{H^{k-1}(\Sigma_0)}  
\end{align*}
 for all $k\geq [\frac{n-1}{2}]+3$. This proves the result for $n\geq 3$. 
 
With minor modifications, the above proof in fact also works for $n=2$. Indeed, we  recall from \eqref{eq:linearized-equation-n=2} that $L_{u_{\mathcal S}}$ has an additional first order term and instead of $g$ we consider $m$ as the background metric. Moreover, instead of $G:= fF$ we set $G :=\frac{F}{ \Psi(u)}$, which  does not change the structure of $G$. 
The analog of \eqref{eq:geometric-energy-estimate} in the $n=2$ case has an additional term on the right hand side which involves the first order term $\int_{\Omega} b^i\partial_iv X(v)$. Analogous to the other terms in the proof for $n\geq 3$, by choosing $a$ sufficiently large, this term will also be absorbed into the left hand side. The commutation estimate in \eqref{eq:commutation-estimate} also holds for $\Box_g$ replaced with $\Box_m + b^i \partial_i$ and, mutatis mutandis, the rest of the proof holds true also in the case $n=2$. 
\end{proof}
\end{proposition}

\subsection{Nash--Moser iteration and the proof of \texorpdfstring{\cref{thm:local-existence}}{Theorem 1}}
With the energy estimate from \cref{lem:energy-estimate} at hand, we will now proceed to the proof of \cref{thm:local-existence}. We note that our estimate in \cref{lem:energy-estimate} suffers from a ``derivative loss'', as we can only control $v$ in $H^k$ using $u$ in $H^{k+2}$. 
 The standard inverse function theorem fails to be applicable in such a scenario; luckily, the more refined Nash--Moser theorem was designed to deal with such a derivative loss. For our purpose, we will use the following version which applies in the smooth category. We refer the reader to \cite{MR656198} for the definitions of the terms involved.
\begin{theorem}[Nash--Moser theorem {\cite[Theorem 1.1.1]{MR656198}}]\label{thm:nash-moser}
Let $\mathcal F$ and $\mathcal G$ be tame spaces and $\Phi\colon U\subset \mathcal F \to \mathcal G$ a smooth tame map. Suppose that the equation $D\Phi(f) h=k$ has a unique solution $h=(D\Phi)^{-1}(f)k$ for all $f \in U$ and all $k\in \mathcal G$, and that the family of inverses $(D\Phi)^{-1}\colon U\times \mathcal G \to \mathcal F$ is a smooth tame map. Then $\Phi$ is locally invertible and each local inverse $\Phi^{-1}$ is a smooth tame map. 
\end{theorem}

This allows us now to prove \cref{thm:local-existence}.

\begin{proof}[Proof of \cref{thm:local-existence}]
We first set $k_0=[\frac{n-1}{2}]+6$. As our tame spaces we  choose 
\begin{align*}&\mathcal F = C^\infty(\Omega) = \cap_{k\geq k_0} H^{k_0}(\Omega)\\
&\mathcal G=C^\infty(\Omega)\times C^\infty(\Sigma_0)\times C^\infty(\Sigma_0) = \cap_{k\geq k_0} H^{k_0}(\Omega) \times  \cap_{k\geq k_0} H^{k_0-1}(\Sigma_0) \times  \cap_{k\geq k_0} H^{k_0-2}(\Sigma_0)
\end{align*}   which are all standard tame spaces, see e.g.\ \cite[p.135]{MR656198}. As our map we choose the nonlinear operator 
\begin{align*}
\Phi \colon U \subset \mathcal F \to \mathcal G, \qquad u \mapsto \Phi(u) := \left( \frac{\det(D^2 u )}{(1+|Du|^2)^{\frac{n+2}{2}}} - K_{\mathcal S}, u\vert_{\Sigma_0}, N_{\mathcal S}  u\vert_{\Sigma_0} \right),
\end{align*}  
where  $U = \{ u \in F \colon \| u-u_{\mathcal S}\|_{H^{k_0}(\Omega)}\}\leq \epsilon_0$. 
In the first component $\Phi$ is a nonlinear differential operator of order $2$, while in the second and third component it is a linear restriction operator that loses one derivative. Thus, $\Phi$ is a tame operator of order $2$, see \cite[Corollary~2.2.7]{MR656198}.

As per usual, the most difficult ingredient to check in using \cref{thm:nash-moser} is  that $(D\Phi)^{-1}$ is a tame map. This is the content of \cref{lem:energy-estimate}. Moreover, with the aid of \cref{lem:energy-estimate} it is easy to check that $(D\Phi)^{-1}\colon U \times \mathcal G \to \mathcal F$ is continuous and hence smooth \cite[Theorem 5.3.1]{MR656198}. Thus, using \cref{thm:nash-moser}, we deduce that $\Phi$ is  smoothly invertible around $u_{\mathcal S}$. 
In particular, for the element $( \eta, u_{\mathcal S}\vert_{\Sigma_0},N_{\mathcal S} u_{\mathcal S}\vert_{\Sigma_0})\in \mathcal G$ with $\eta$ belonging to a sufficiently small neighborhood of zero, there exists $u_\eta$ such that $\Phi(u_\eta) =\eta$ and satisfying $u_\eta=u_{\mathcal S}$ and $N_{\mathcal S} u_\eta= N_{\mathcal S}u_{\mathcal S}$ on $\Sigma_0$. In particular, this implies that $Du_\eta = Du_{\mathcal S}$ on $\Sigma_0$, concluding the proof.
\end{proof}
\section{Beyond the local analysis}
\subsection{A criterion for fully localized solutions}

In \cref{thm:local-existence}  we showed that, for all sufficiently small  perturbations $\eta\in C^\infty(\Omega)$, we can solve the prescribed Gauss curvature equation \eqref{eq:prescribed-curvature-equation} and obtain a perturbation $u$ of the original surface $u_{\mathcal S}$ that agrees with $u_{\mathcal S}$ on the initial slice $\Sigma_0$. A natural question is whether, for  perturbations that vanish of the boundary of $\Omega$ (i.e.,   $\eta \in C_c^\infty(\operatorname{int}(\Omega))$), one can find perturbations $u$ of $u_{\mathcal S}$ that  agree with $u_{\mathcal S}$ on the whole boundary of $\Omega$. 

An example of such perturbations can be constructed ``tautologically'' by considering a family of graphs $u$ that are close to $u_{\mathcal S}$ and only differ from $u_{\mathcal S}$ outside a set compactly supported inside $\operatorname{int}(\Omega)$, and then defining $\eta:=\frac{\det(D^2 u)}{(1+|Du|^2)^{\frac{n+2}{2}}} - K_{\mathcal S} $.

 In the following, given a smooth one parameter family of perturbations of the curvature $\eta_s$, we will derive a necessary condition for this family to give rise to a smooth one parameter family of solutions $u_s$ with the property that $u_s$ agrees with $u_{\mathcal S}$ outside a compact set  contained in $\Omega$. Unlike in \cref{thm:local-existence} we shall merely assume that $\Omega\subset \mathbb R^n$ is a smooth (possibly unbounded) domain.   Moreover, for the linearized equation, we  show that this condition (together with global hyperbolicity of $\Omega$, so that the linearized equation is well-posed) is in fact sufficient to construct a solution $v$ of the linearized equation with the property that $v$ is supported in a compact set contained inside $\Omega$. We will state the following proposition only for $n\geq 3,$ but we note that for $n=2$ an analogous statement can be shown. 

\begin{proposition}\label{prop:full-localization}
	Let $n\geq 3$ and let $\Omega \subset \mathbb R^n$ be a (possibly unbounded) open domain.  Let $\eta_s\in C^\infty(\Omega)$ be a smooth 1-parameter family of perturbations parameterized by $s\in [0,1]$ such that $\cup_{s\in [0,1]}\operatorname{supp}(\eta_s) \subset \Omega $ and $\eta_0 =0$. Assume that there exist associated solutions $u_{\eta_s}$ to \eqref{eq:prescribed-curvature-equation}  which are fully localized (in the sense that  $\operatorname{supp}(u_{\eta_s}-u_{\mathcal S}) \subset \Omega $) for all $s\in [0,1]$. Then, for all $s\in [0,1]$,
	\begin{align}
		\int_{\Omega} w  f_s  \dot \eta_s  \,\, \textup{vol}_{g_{u_{\eta_s}}} =0\qquad \text{for all $w \in C^\infty(\Omega)$ satisfying $\Box_{g_{u_{\eta_s}}} w=0$,} \label{eq:orthonality-condition-1}
	\end{align}
where
	$$
	f_s= -  | \det D^2u_{\eta_s}|^{- \frac{n-1}{n-2}} ( 1+ |Du_{\eta_s}|^2)^{ \frac{n(n+2)}{2(n-2)}}\qquad \text{and}\qquad \dot \eta_s = \frac{\d }{\d s } \eta_s
	$$
	(cp. \eqref{eq:definition-of-f}).	 In other words, $f_s \dot \eta_s$ is orthogonal to the kernel of $\Box_{g_{u_{\eta_s}}}$ for each $s\in [0,1]$. 
	
	Conversely, assume in addition that $\Omega$ is globally hyperbolic.
	If $\tilde \eta\in C^\infty_c(\Omega)$ satisfies \begin{align}
		\int_{\Omega}  w  f  \tilde \eta   \,\, \textup{vol}_{g_{u_{\mathcal S}}} =0
		\qquad \text{for  all $w \in C^\infty(\Omega)$ satisfying $\Box_{g_{u_{\mathcal S}}} w=0$}
		 \label{eq:orthonality-condition-2}
	\end{align} with $f$ as in \eqref{eq:definition-of-f}, then there exists a solution $v\in C^\infty(\Omega)$ to the linearized Gauss curvature equation 
	\begin{align}\label{eq:linearized-equation-tilde-eta}
		L_{u_{\mathcal S}} v = \tilde \eta
	\end{align}
	with the property that $v$ is fully localized, in the sense that $\operatorname{supp}(v)\subset   \Omega $. 
\end{proposition}
\begin{proof}
	For the first claim we differentiate \eqref{eq:prescribed-curvature-equation} with respect to $s$ so that $v_s:= \frac{\d}{\d s} u_{\eta_s}$ satisfies the linearized equation   \eqref{eq:linearized-equation} with $\dot \eta_s$ as right hand side, that is
	\begin{align*}
\frac{1}{f_s} \Box_{g_{u_{\eta_s}}} v_s = \dot\eta_s
	\end{align*}
(see \cref{prop:equation-as-geometric-wave-equation}), or equivalently
$\Box_{g_{u_{\eta_s}}} v = f_s \dot \eta_s$. To obtain \eqref{eq:orthonality-condition-1}, we multiply this equation by $w$ and integrate. Using that $\Box_{g_{u_{\eta_s}}} $ is a symmetric operator with respect to the volume form induced by $g_{u_{\eta_s}}$, and using that $\Box_{g_{u_{\eta_s}}} w =0$, we obtain \eqref{eq:orthonality-condition-1}. Note that we also used that the support of $v$ is contained in $\Omega $ to ensure that no boundary terms appear in the integration by parts.
	
	For the second claim, let  $\tilde \eta$ satisfy \eqref{eq:orthonality-condition-2}, and let $v\in C^\infty(\Omega)$ be the unique solution to \eqref{eq:linearized-equation-tilde-eta} arising from vanishing data on a Cauchy hypersurface $\Sigma_-$ that lies to the past of the support of $\tilde \eta$. We note that such a hypersurface always exists as $\Omega$ is globally hyperbolic and $\tilde \eta$ is compactly supported.  By the domain of dependence, we remark that $\operatorname{supp}(v) \subset J^+(\operatorname{supp}(\tilde \eta))$, where $J^+(X)$ (resp. $J^-(X)$) denotes the causal future (resp. causal past) of the set $X$ with respect to the Lorentzian metric $g$.\footnote{We recall that $p\in J^\pm(X)$ if either $p\in X$ or if there exists a future/past directed causal curve starting in $X$ and terminating at $p$. Similarly, we recall that $p \in I^\pm(X)$  if there exists a future/past directed timelike curve starting in $X$ and terminating at $p$.} 

	 \begin{figure}
		\centering
		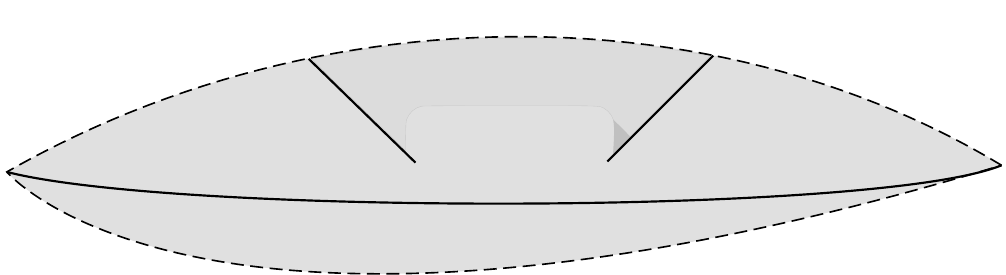
		\caption{Illustration of some of the regions involved in the proof of \cref{prop:full-localization}. The darker shaded regions represent  $J^+(\operatorname{supp}(\tilde \eta ))$ and $J^-(\operatorname{supp}(\tilde \eta ))$, respectively.}
		\label{fig:prop4-illustration}
	\end{figure}

	  To show that $\operatorname{supp}(v) \subset J^-(\operatorname{supp}(\tilde \eta))$, let $\varphi$ be an arbitrary smooth test function supported inside $J^+(\operatorname{supp}(\tilde \eta)) \setminus J^-(\operatorname{supp}(\tilde \eta))$, fix $\Sigma_+$ a Cauchy hypersurface to the future of  the support of $\varphi$ (cf.\ \cref{fig:prop4-illustration}), and consider $\tilde w\in C^\infty(\Omega)$ the unique solution to  \begin{align*}
		\Box_g \tilde w = \varphi 
	\end{align*}
	with trivial data on $\Sigma_+$. Then $\operatorname{supp}(\tilde w) \subset J^-(\operatorname{supp}(\varphi) )$. 
	Moreover, let $  w \in C^\infty(\Omega)$ be the solution to
	$$
	\Box_g   w = 0 ,\qquad  \text{with }  w = \tilde w, N_{\Sigma_-}w =  N_{\Sigma_-} \tilde w \text{ on the Cauchy hypersurface $\Sigma_-$.}
	$$
	In this way, by uniqueness for the wave equation, $w$ actually satisfies
	\begin{align*}
		\Box_g   w = 0,\qquad  \text{with }  w = \tilde w \text{ for }  \Omega \setminus J^+(\operatorname{supp}(\varphi)).
	\end{align*}
With these definitions, integrating by parts we obtain
	\begin{align}
		\int \varphi v  \operatorname{dvol}_{g} = \int (\Box_g \tilde w  )v \operatorname{dvol}_{g} = \int \tilde w  (\Box_g v) \operatorname{dvol}_{g} = \int \tilde w f  \tilde {\eta}  \operatorname{dvol}_{g}.
		\label{eq:int parts}
	\end{align}
	Now, we note that if $p \in \operatorname{supp}(\tilde \eta)\subset J^-(\operatorname{supp}(\tilde \eta))$, then $p \notin \operatorname{supp}(\varphi) \subset J^+(\operatorname{supp}(\varphi))$.  Thus, $w = \tilde w $ on the support of $\tilde \eta$. Hence, by \eqref{eq:int parts} and \eqref{eq:orthonality-condition-2}
	we get
	\begin{align*}
		\int \varphi v  \operatorname{dvol}_g   = \int    w f \tilde{\eta}  \operatorname{dvol}_g =0.
	\end{align*}
This shows that $\operatorname{supp}(v)\subset J^+(\operatorname{supp}(\tilde \eta)) \cap J^-(\operatorname{supp}(\tilde \eta))$. Finally, to see that $J^+(\operatorname{supp}(\tilde \eta)) \cap J^-(\operatorname{supp}(\tilde \eta))$  is a compact subset of $\Omega$, we cover the compact set  $\operatorname{supp}(\tilde \eta) $ with finitely many open sets of the form $I^+(p_i)\cap I^-(q_i) $ with $p_i,q_i\in \Omega $, $1\leq i \leq M$. In particular, it follows that $ J^+(\operatorname{supp}(\tilde \eta)) \cap J^-(\operatorname{supp}(\tilde \eta))\subset \bigcup_{1\leq i,j\leq M} J^+(p_i)\cap J^-(q_j)$. As $\Omega$ is globally hyperbolic, we have that the sets $J^+(p_i)\cap J^-(q_j)$ are compact, from which the claim follows.
\end{proof}

\subsection{A linear instability result}
In the previous section, we found  a natural necessary condition for the existence of localized solutions to the  linearized Gauss curvature equation (cf.~\eqref{eq:orthonality-condition-1}), that holds in particular for $\Omega=\mathbb R^n$. 
In particular, this shows that for generic arbitrarily small and localized perturbation of the curvature, one cannot expect the perturbation of the surface to be compactly supported unless \eqref{eq:orthonality-condition-1} is satisfied.

Still, given arbitrarily small and localized perturbations of the curvature, one may wonder whether there exist  solutions that remain close (in a suitable norm) to the original solution $u_{\mathcal S}$ at infinity, even if \eqref{eq:orthonality-condition-1} fails. At the linear level, this would correspond to understanding whether the   graph $u_{\mathcal S}$ is \emph{linearly stable}, in the sense that solutions to the linearized equations arising from smooth and localized initial data are asymptotically bounded, or at least grow slower than $u_{\mathcal S}.$
As we shall show now, linear stability may fail in general.

Consider the hyperbolic paraboloid $\mathcal S\subset \mathbb R^{n+1}$ given as the graph of the function
 \begin{align*}  u_{\mathcal S}(t,x)  = \frac 12 ( |x|^2 - t^2 ),\end{align*} where $x = (x_1,\dots, x_{n-1})  \in \mathbb R^{n-1}$ and $t \in \mathbb R$. Its Hessian is given by
\begin{align*}
	D^2 u_{\mathcal S} = \textup{diag} ( -1, 1 , \dots , 1),
\end{align*}
thus $\det D^2 u_{\mathcal S} = - 1 $, and the Gauss curvature of $\mathcal S$ is
\begin{align*}
	K_{\mathcal S} (t,x) = \frac{\det D^2 u_{\mathcal S} }{(1 + |D u_{\mathcal S} |^2 )^{\frac{n+2}{2} }} = \frac{-1}{ ( 1 + |x|^2 + t^2)^{\frac{n+2}{2} }}.
\end{align*}
In particular, the linearized operator reads
\begin{align}\label{eq:linearized-operator-hyperbolic-paraboloid}
K_{\mathcal S}^{-1} L_{u_{\mathcal S}}(v):= 	-\partial_{t}^2v + \Delta_x v- (n+2) \frac{-t \partial_t v + x\cdot \nabla_x v}{1+ t^2 + |x|^2 } 
\end{align}

For simplicity we focus on the case $n=2$ and show that the  operator \eqref{eq:linearized-operator-hyperbolic-paraboloid} is  linearly unstable. We note that, for higher dimensions $n\geq 3$, one could also show similar instability results. 
\begin{proposition} Let $n=2.$
Then,  the linearized prescribed Gauss curvature problem is linearly unstable around the hyperbolic paraboloid. More precisely, there exists smooth and compactly supported data posed on $\{ t=0\}$ for  the equation
\begin{align}\label{eq:linearized-equation-hyperbolic-paraboloid}
	-\partial_{t}^2v + \partial_x^2 v- 4 \frac{-t \partial_t v + x \partial_x v}{1+ t^2 + |x|^2 } =0
\end{align}
 such that the arising solution $v$ satisfies
\begin{align}\label{eq:unbounded-growth}
\sup_{x \in \mathbb R} |v(t,x)| \to \infty
\end{align}
as $t\to\infty$. 
\begin{proof}
We use double null coordinates $\underline{\zeta} =t+x$ and ${\zeta}=t-x$, and we write \eqref{eq:linearized-equation-hyperbolic-paraboloid} as 
\begin{align}\label{eq:linearized-equation-hyperbolic-paraboloid-double-null} 
 \partial_{\underline{\zeta}} \partial_{ {\zeta} }v =  \frac{ {\zeta} \partial_{\underline{\zeta}} v + \underline{\zeta} \partial_{\zeta} v}{1+ \frac{\underline{\zeta}^2}{2}+ \frac{{\zeta}^2}{2}}.
\end{align}
We will now pose characteristic data on $\{ {\zeta} =0, \underline{\zeta} \geq 0\} \cup \{ {\zeta} \geq 0, \underline{\zeta}=0\}$.

On $\{ \underline{\zeta} =0,  {\zeta} \geq 0\}$ we set $v(\underline{\zeta}=0,  {\zeta}) :=  {\zeta} \chi( {\zeta})$, where $\chi\colon \mathbb R \to [0,1]$ is a smooth cut-off satisfying $\chi(x) =1$ for $x\leq 1$ and $\chi(x) =0$ for $x\geq 2$. On  $\{ \underline{\zeta} \geq 0,  {\zeta}=0\}$ we set $v\equiv 0$. Note that these constitute smooth characteristic data for \eqref{eq:linearized-equation-hyperbolic-paraboloid}.

The initial data satisfy $\partial_{ {\zeta}} v = 1 \geq 0$ for $ {\zeta} \in [0,1]$ and $\underline{\zeta}=0$, and  $\partial_{\underline{\zeta}} v = 0$ on $\underline{\zeta} \geq 0$ and ${\zeta}=0$.  
One can directly observe that the equation \eqref{eq:linearized-equation-hyperbolic-paraboloid-double-null} propagates this behavior in the region $\mathcal C:= \{0\leq   {\zeta} \leq 1\} \times \{\underline{\zeta} \geq 0\} $. 
In particular, in the region $\mathcal C$ we have the inequality
\begin{align*}
\partial_{\underline{\zeta}} (\partial_{\zeta} v) \geq \frac{2\underline{\zeta}}{ 3 + {\zeta}^2} \partial_{\zeta} v. 
\end{align*}
Integrating in $\underline{\zeta}$, this yields 
\begin{align*}
\log \partial_{\zeta} v (\underline{\zeta}, {\zeta} )  \geq \int_0^{\underline{\zeta}} \frac{2 \tilde{\underline{\zeta}}} {3 + \tilde{\underline{\zeta}}^2} \d \tilde {\underline{\zeta}} = \log\left(1+ \frac{\underline{\zeta}^2}{3} \right)
\end{align*}
so in particular
\begin{align*}
\partial_{\zeta} v (\underline{\zeta}, {\zeta} ) \geq \frac 13 \underline{\zeta}^2.
\end{align*}
Integrating in ${\zeta}$ and using the initial data finally gives 
\begin{align*}
v (\underline{\zeta}, \zeta ) \geq  \frac 13 \zeta \underline{\zeta}^2
\end{align*}
in the region $\mathcal C$. Going back to $(t,x)$ coordinates, we obtain \eqref{eq:unbounded-growth}. From our construction it is clear that $v$ can also be realized as arising from smooth and compactly supported data on $\{t=0\}$.
\end{proof}
\end{proposition}

\begin{remark}
As one can see from the construction above, not only $v$ is unbounded but actually grows faster than $u_{\mathcal S}$ itself. More precisely,
$$
v(t,0)\geq \frac13 t^3 \gg |u_{\mathcal S}(t,0)| \qquad \forall\,t\gg 1,
$$
so $u_{\mathcal S}+\epsilon v$ is never a perturbation of $u_{\mathcal S}$ for $\epsilon \neq 0$.
\end{remark}

\section{Conclusion and perspectives}
Having established the local solvability of \eqref{eq:Gauss-curvature-equation} near a given surface (cf. Theorem~\ref{thm:local-existence}), and having found a natural necessary condition for the existence of localized solutions to the  linearized Gauss curvature equation (cf. \cref{prop:full-localization}), a natural important open question is now the following: 

\smallskip

\noindent
{\it Given a \emph{global} graph $u_{\mathcal S}\colon \mathbb R^n \to \mathbb R$ with Lorentzian Hessian, under what conditions there exists a nearby surface for small and fast decaying perturbations of the Gauss curvature?}

\smallskip

As a preliminary step, also in view of the strategy employed to prove Theorem~\ref{thm:local-existence}, one should analyse the linear problem, and try to understand the linear stability of  $u_{\mathcal S}$.
A sufficient condition for compactly supported perturbation is given by \eqref{eq:orthonality-condition-2}, but one may like to understand whether such condition suffices to guarantee that fast-decaying perturbation of the curvature induce fasts-decaying solutions to the linear problem, or whether more assumptions are needed.

Overall, this problem has a very rich structure, and it would be extremely interesting to find sufficient/necessary conditions for the solvability of the linear/global problem with natural decay properties on the solutions.

\bigskip

\noindent
{\it Acknowledgments.} The first author is supported by the European Research Council under the Grant Agreement No. 721675 ``Regularity and Stability in Partial Differential Equations (RSPDE).'' The second author acknowledges support by a grant from the Institute for Advanced Study, by Dr.~Max R\"ossler, the Walter Haefner Foundation, and the ETH Z\"urich
Foundation.

\printbibliography
\end{document}

%% file: omega.pdf_tex
\begingroup%
  \makeatletter%
  \providecommand\color[2][]{%
    \errmessage{(Inkscape) Color is used for the text in Inkscape, but the package 'color.sty' is not loaded}%
    \renewcommand\color[2][]{}%
  }%
  \providecommand\transparent[1]{%
    \errmessage{(Inkscape) Transparency is used (non-zero) for the text in Inkscape, but the package 'transparent.sty' is not loaded}%
    \renewcommand\transparent[1]{}%
  }%
  \providecommand\rotatebox[2]{#2}%
  \newcommand*\fsize{\dimexpr\f@size pt\relax}%
  \newcommand*\lineheight[1]{\fontsize{\fsize}{#1\fsize}\selectfont}%
  \ifx\svgwidth\undefined%
    \setlength{\unitlength}{203.73461357bp}%
    \ifx\svgscale\undefined%
      \relax%
    \else%
      \setlength{\unitlength}{\unitlength * \real{\svgscale}}%
    \fi%
  \else%
    \setlength{\unitlength}{\svgwidth}%
  \fi%
  \global\let\svgwidth\undefined%
  \global\let\svgscale\undefined%
  \makeatother%
  \begin{picture}(1,0.2836658)%
    \lineheight{1}%
    \setlength\tabcolsep{0pt}%
    \put(0,0){\includegraphics[width=\unitlength,page=1]{omega.pdf}}%
    \put(0.31220511,0.00538509){\color[rgb]{0,0,0}\makebox(0,0)[lt]{\lineheight{1.25}\smash{\begin{tabular}[t]{l}$\Sigma_0$\end{tabular}}}}%
    \put(0.50186844,0.25484981){\color[rgb]{0,0,0}\makebox(0,0)[lt]{\lineheight{1.25}\smash{\begin{tabular}[t]{l}$\Sigma_1$\end{tabular}}}}%
    \put(0.47887403,0.16647768){\color[rgb]{0,0,0}\makebox(0,0)[lt]{\lineheight{1.25}\smash{\begin{tabular}[t]{l}$\Sigma_t$\end{tabular}}}}%
    \put(0.47887403,0.16647768){\color[rgb]{0,0,0}\makebox(0,0)[lt]{\lineheight{1.25}\smash{\begin{tabular}[t]{l}$\Sigma_t$\end{tabular}}}}%
    \put(0.86513066,0.19548396){\color[rgb]{0,0,0}\makebox(0,0)[lt]{\lineheight{1.25}\smash{\begin{tabular}[t]{l}$\Xi_1$\end{tabular}}}}%
    \put(0.05097084,0.15590678){\color[rgb]{0,0,0}\makebox(0,0)[lt]{\lineheight{1.25}\smash{\begin{tabular}[t]{l}$\Xi_1$\end{tabular}}}}%
  \end{picture}%
\endgroup%

%% file: prop4-illustration.pdf_tex
\begingroup%
  \makeatletter%
  \providecommand\color[2][]{%
    \errmessage{(Inkscape) Color is used for the text in Inkscape, but the package 'color.sty' is not loaded}%
    \renewcommand\color[2][]{}%
  }%
  \providecommand\transparent[1]{%
    \errmessage{(Inkscape) Transparency is used (non-zero) for the text in Inkscape, but the package 'transparent.sty' is not loaded}%
    \renewcommand\transparent[1]{}%
  }%
  \providecommand\rotatebox[2]{#2}%
  \newcommand*\fsize{\dimexpr\f@size pt\relax}%
  \newcommand*\lineheight[1]{\fontsize{\fsize}{#1\fsize}\selectfont}%
  \ifx\svgwidth\undefined%
    \setlength{\unitlength}{289.63814503bp}%
    \ifx\svgscale\undefined%
      \relax%
    \else%
      \setlength{\unitlength}{\unitlength * \real{\svgscale}}%
    \fi%
  \else%
    \setlength{\unitlength}{\svgwidth}%
  \fi%
  \global\let\svgwidth\undefined%
  \global\let\svgscale\undefined%
  \makeatother%
  \begin{picture}(1,0.27308764)%
    \lineheight{1}%
    \setlength\tabcolsep{0pt}%
    \put(0,0){\includegraphics[width=\unitlength,page=1]{prop4-illustration.pdf}}%
    \put(0.17739994,0.04614872){\color[rgb]{0,0,0}\makebox(0,0)[lt]{\lineheight{1.25}\smash{\begin{tabular}[t]{l}$\Sigma_-$\end{tabular}}}}%
    \put(0,0){\includegraphics[width=\unitlength,page=2]{prop4-illustration.pdf}}%
    \put(0.27879234,0.26829793){\color[rgb]{0,0,0}\makebox(0,0)[lt]{\lineheight{1.25}\smash{\begin{tabular}[t]{l}$\mathrm{supp}(\varphi)$\end{tabular}}}}%
    \put(0,0){\includegraphics[width=\unitlength,page=3]{prop4-illustration.pdf}}%
    \put(0.76733529,0.14540355){\color[rgb]{0,0,0}\makebox(0,0)[lt]{\lineheight{1.25}\smash{\begin{tabular}[t]{l}$\Sigma_+$\end{tabular}}}}%
    \put(0,0){\includegraphics[width=\unitlength,page=4]{prop4-illustration.pdf}}%
    \put(0.43672621,0.13157241){\color[rgb]{0,0,0}\makebox(0,0)[lt]{\lineheight{1.25}\smash{\begin{tabular}[t]{l}$\mathrm{supp}(\tilde \eta)$\end{tabular}}}}%
    \put(0,0){\includegraphics[width=\unitlength,page=5]{prop4-illustration.pdf}}%
  \end{picture}%
\endgroup%